\documentclass[10pt]{amsart}

\usepackage{latexsym, graphicx}
\usepackage{amsmath}
\usepackage{amsfonts}
\usepackage{amssymb}
\newcommand{\be}{\begin{equation}}
\newcommand{\ee}{\end{equation}}
\newcommand{\rfb}[1]{\mbox{\rm
   (\ref{#1})}\ifx\undefined\stillediting\else:\fbox{$#1$}\fi}
\newcommand{\half}   {{\frac{1}{2}}}

\usepackage{latexsym,amssymb}
\usepackage{enumerate}
\newtheorem{Theorem}{Theorem}[section]

\newtheorem{Proposition}[Theorem]{Proposition}
\newtheorem{Lemma}[Theorem]{Lemma}
\newtheorem{Corollary}[Theorem]{Corollary}

\newtheorem{Remark}[Theorem]{Remark}

\numberwithin{equation}{section}

\usepackage{amsfonts,layout}

\begin{document}

\pagenumbering{arabic}



\setcounter{secnumdepth}{3}


\renewcommand{\theequation}{\thesection.\arabic{equation}}
\newcommand{\snl}{\ \\}
\newcommand{\pagref}[1]{page~\pageref{#1}}
\newcommand{\appref}[1]{appendix~\ref{#1}}
\newcommand{\chapref}[1]{chapter~\ref{#1}}
\newcommand{\tabref}[1]{table~\ref{#1}}
\newcommand{\figref}[1]{figure~\ref{#1}}
\newcommand{\secref}[1]{section~\ref{#1}}
\newcommand{\subsecref}[1]{subsection~\ref{#1}}
\newcommand{\ltfigure}[2]{\caption{#2} \label{#1}}

\newcommand{\m}[1]{{\(#1\)}}
\newcommand{\tilda}{\m{\sim}}
\newcommand{\mtilda}{\m{\sim}}
\newcommand{\msub}[2]{{\({#1}_{#2}\)}}
\newcommand{\msuper}[2]{{\({#1}^{#2}\)}}
\newcommand{\ucn}[1]{Ultracomputer Note \#{#1}}
\newcommand{\quotetitle}[1]{{\em ``#1''}}
\newcommand{\mla}{{$<=\ $}}
\newcommand{\smla}{{$<=$}}
\newcommand{\cons}{{$ \mid\ $}}

\newcommand{\samel}[1]{\mbox{#1}}
\newcommand{\ems}[1]{\samel{\em #1}}

\newcommand{\mypar}[1]{\paragraph{#1}\ \newline}

\newcommand{\rubsp}{\hspace{0em plus .25em}}
\newcommand{\rubspa}{\hspace{0em plus .1em}}


%

%
%
%
%

\def\squareforqed{\rule{1.5ex}{1.5ex}}

\def\qed{%
        \ifmmode\squareforqed\else{\unskip\nobreak\hfil%
        \penalty50\hskip1em\null\nobreak\hfil\squareforqed%
        \parfillskip=0pt\finalhyphendemerits=0\endgraf}\fi%
}

\let\finpreuve=\qed


\newcommand{\Nu}{\Upsilon}

\bibliographystyle{plain}

\thispagestyle{empty}
\title[Stabilization of the nonlinear damped wave equation]{Stabilization of the nonlinear damped wave equation via linear weak observability}

\author{Ka\"{\i}s Ammari}
\address{UR Analysis and Control of PDE, UR13ES64, Department of Mathematics, Faculty of Sciences of Monastir, University of Monastir, 5019 Monastir, Tunisia} 
\email{kais.ammari@fsm.rnu.tn} 

\author{Ahmed Bchatnia}
\address{UR  Analyse Non-Lin\'eaire et G\'eom\'etrie, UR13ES32, Department of Mathematics, Faculty of Sciences of Tunis, University of Tunis El Manar, Tunisia}
\email{ahmed.bchatnia@fst.rnu.tn} 

\author{Karim El Mufti}
\address{UR Analysis and Control of PDE, UR13ES64, ISCAE, University of Manouba, Tunisia}
\email{karim.elmufti@iscae.rnu.tn}

\begin{abstract}
We consider the problem of energy decay rates for nonlinearly damped abstract infinite dimensional systems. We prove sharp, simple and quasi-optimal energy decay rates through an indirect method, namely a weak observability estimate for the corresponding undamped system. One of the main advantage of these results is that they allow to combine the optimal-weight convexity method of \cite[Alabau-Boussouira]{amo2005} and a methodology of \cite[Ammari-Tucsnak]{ammari} for weak stabilization by observability. Our results extend to nonlinearly damped systems, those of Ammari and Tucsnak \cite{ammari}. At the end, we give an appendix on the weak stabilization of linear evolution systems.
\end{abstract}

\maketitle

\tableofcontents


\section{Introduction} \label{intro}
We consider the following second order differential equation
\begin{equation}\label{damped1}
\begin{cases}
 \ddot{ w}(t) + Aw (t) + a(.) \rho(., \dot w)=0\,,
\qquad t\in (0,\infty)\,, x \in \Omega \\
w(0)=w^0 \,, \dot w (0)=w^1 \,.
\end{cases}
\end{equation}
where $\Omega$ is a bounded open set in $\mathbb{R}^N$, with a boundary $\Gamma$ and $\rho : \overline{\Omega}\times\mathbb{R} \rightarrow \mathbb{R}$ is supposed to be continuous on $\overline{\Omega}\times\mathbb{R}$ and strictly monotone with respect to the second variable. We assume that $\Omega$ is either convex or of class $\mathcal{C}^{1,1}$. We set $H=L^2(\Omega)$, with its usual scalar product denoted by $\langle\cdot,\cdot\rangle_H$ and the associated norm $\|\cdot\|_H$ and where $A:D(A)\subset H\to H$ is a densely defined self-adjoint linear operator  satisfying
\begin{equation}\label{eq:opA}
\langle Au,u\rangle_H\ge C \|u\|_H^2\qquad\forall u\in D(A)
\end{equation}
for some $C>0$. We also introduce the scale of Hilbert spaces
$H_{\alpha}$, as follows\m: for every $\alpha\geq 0$,
$H_{\alpha}=\mathcal{D}(A^{\alpha})$, with the norm $\|z
\|_{\alpha}=\|A^\alpha z\|_{H}$. The space $H_{-\alpha},$ is defined
by duality with respect to the pivot space $H$ as follows\m:
$H_{-\alpha} =H_{\alpha}^*,$ for $\alpha>0$. The operator $A$ can be
extended (or restricted) to each $H_{\alpha}$, such that it becomes
a bounded operator
\begin{equation}\label{A0ext}
A : H_{\alpha} \rightarrow H_{\alpha-1} \quad \, \forall \  \alpha
\in \mathbb{R} \,.
\end{equation}
\begin{description}
\item[Assumption (A1)]  \
There exists a continuous strictly increasing odd function $g \in \mathcal{C}([-1,1];\mathbb{R})$, continuously differentiable in a neighbourhood of $0$ and satisfying $g(0)=g^{\prime}(0)=0$, with
\begin{equation}\label{eqrhoat0}
\begin{cases}
c_1 g(|v|) \le  |\rho(.,v)| \le c_2 g^{-1}(|v|)
\,, \quad |v| \le 1 \,, \mbox{ a.e. on } \Omega \,,\\
c_1 |v| \le  |\rho(.,v)| \le c_2 |v| \,, \quad |v| \ge 1 \,, \mbox{
a.e. on } \Omega \,,
\end{cases}
\end{equation}
where $g^{-1}$ denotes the inverse function of $g$ and $c_i >0$ for $i=1,2$. Moreover $a \in
\mathcal{C}(\overline{\Omega})$, with $a \ge 0$ on $\Omega$ and there exists $a_0 > 0$ such that $a(x) \geq a_0$ on $\omega$. Here $\omega$ stands for the subregion of $\Omega$ on which the feedback $\rho$ is active.
\end{description}

\medskip
The equation \eqref{damped1} is understood as an equation in
$H_{-1/2}$, i.e., all the terms are in $H_{-1/2}$. The energy of a
solution is defined by

\begin{equation}\label{energydef}
E_w(t) = \frac{1}{2} \Big( \|(w(t),\dot w(t))\|^2_{H_{1/2} \times H}
\Big)
\end{equation}
\noindent Most of the nonlinear equations modelling the damped
vibrations of elastic structures can be written in the form
\eqref{damped1}, where $w$ stands for the displacement field and the
term $B\dot w(t)=a(.)\rho(.,\dot w)$, represents a viscous feedback
damping. 

Let us introduce the operator 
$$
{\mathcal A} = \left(
\begin{array}{cc}
  0 & I \\
- A & - a \rho
\end{array}
\right) : D({\mathcal A}) = H_1 \times H_{1/2} \subset H_{1/2} \times H \rightarrow H_{1/2} \times H
$$
and \eqref{damped1} becomes 
$$
\dot{W} = {\mathcal A} W, \, W(0) = W^0,
$$
where
$W^0 = \left( \begin{array}{ll} w^0 \\ w^1 \end{array} \right)$ and $W = \left( \begin{array}{ll} w \\ \dot{w} \end{array} \right).$

The operator ${\mathcal A}$ is the generator of a continuous semigroup of nonlinear contractions in $H_{1/2} \times H$ (see \cite[Corollary 2.1, page 35]{barbu}). Then the system \eqref{damped1} is well-posed. More precisely, the following holds:

If $(w^0,w^1)\in H_{1} \times H_{1/2}$. Then the problem
\eqref{damped1} admits a unique strong solution
$$w \in C([0,\infty);H_{1}) \cap C^{1}([0,\infty);H_{1/2}).$$ 
Moreover, if $(w^0,w^1) \in H_{1/2} \times H$ then the system \eqref{damped1} 
admits a unique mild solution, i.e., $(w,\dot{w}) \in C([0,+\infty),H_{1/2} \times H)$. 

We have for all $t\geq 0$, the following energy identity:

\begin{equation}\label{esten}
\|(w^0,w^1)\|^2_{H_{1/2} \times H} - \|(w(t),\dot w(t))\|^2_{H_{1/2}
\times H}= 2 \, \int_0^t \int_{\Omega} a(.) \rho(.,\dot{w}(s)),\dot
w(s)\,dx \,ds.
\end{equation}
The aim of this paper is to deduce energy decay rates from weak
observability estimates for the associated undamped system, that is

\begin{equation}\label{conservative}
\begin{cases}
\ddot \phi(t) + A \phi(t) = 0,  \\
\phi(0) = \phi^0, \, \dot \phi(0) = \phi^1.
\end{cases}
\end{equation}
Our results extend to nonlinearly damped systems, those of Ammari
and Tucsnak~\cite{ammari} (see also \cite{ammariniciase} for more details) which concern linearly damped systems.
\section{Preliminaries and main results} \label{trans}
\medskip

Before stating our main results, let us precise some hypotheses on
the feedback and give some preliminary definitions.

\medskip
We define a function $R$ (see~\cite{amo2005})  by
\begin{equation}\label{H}
R(x) =\sqrt{x}g(\sqrt{x}) \,, \ x \in [0,r_0^2] \,,
\end{equation}
Thanks to assumption ${\bf{(A1)}}$, $R$ is of class $\mathcal{C}^1$ and is
strictly convex on $[0,r_0^2]$, where $r_0>0$ is a sufficiently
small number. We still denote by $R$ its extension to $\mathbb{R}$
with $R(x)=+\infty$ for $x \in \mathbb{R} \backslash [0,r_0^2]$. We
also define a function $L$ by
\begin{equation}\label{L}
L(y)=\left\{
\begin{array}{l}
\displaystyle{\frac{R^{\star}(y)}{y} \quad \mbox{ , if  } y \in (0,
+\infty) \,,}\\ [0.2cm] 0 \quad \mbox{ , if       } y =0 \,,
\end{array}\right.
\end{equation}
\noindent where $R^{\star}$ stands for the convex conjugate function
of $R$, i.e.: $ R^{\star}(y)=\sup_{x \in \mathbb{R}}\{xy - R(x)\}.$
Moreover we define a weight function $f$ such that
\begin{equation}\label{weight}
R^{\star}(f(s))= \frac{s f(s)}{\beta} \,, \quad \ s \in [0, \beta
r_0^2) \,,
\end{equation}
\noindent where $\beta$ is a constant that will be chosen later. We
recall that f is defined by
$$
f(s)=L^{-1}\Big(\frac{s}{\beta}\Big) \,, \quad \forall \ s \in [0,
\beta r_0^2) \,.
$$
\noindent One can show \cite{amo2005} that $f$ is a strictly
increasing function from $[0, \beta r_0^2)$ onto $[0, \infty)$.\\

After, we consider the unbounded operator
\begin{equation} \label{op}
{\mathcal A}_d : {\mathcal D}({\mathcal A}_d) \subset H_{1/2} \times H \rightarrow H_{1/2} \times H, \,
{\mathcal A}_d = \left(
\begin{array}{cc}
0 &I \\
- A &- a\\
\end{array}
\right),
\end{equation}
where
$$
{\mathcal D}({\mathcal A}_d) = H_1 \times H_{1/2}.
$$

Let $X_1, X_2$ be two Banach spaces such that
$$
{\mathcal D}({\mathcal A}_{d}) \subset H_{1/2} \times
H \subset X_1 \times X_2, \,
$$
with continuous embeddings
and
\begin{equation} 
\label{h3}
[H_1 \times H_{1/2}, X_1 \times X_2 ]_{\theta} = H_{1/2} \times H,
\end{equation}
for a fixed real number $0<\theta <1$, where $[. ,. ]_{\theta}$
denotes the interpolation space (see for instance Triebel
\cite{triebel}, \cite{ammari}) and ${\mathcal G}: \mathbb{R}_+
\rightarrow \mathbb{R}_+$ be an increasing and continuous function on ${\mathbb R}_+ = (0,\infty)$. 
\begin{description}
\item[Assumption (A2)]  \
There exist $T, C_T >0$ such that the following observability inequality
is satisfied for the linear conservative system \eqref{conservative} 

\begin{equation}\label{ineq-obs}
c_T E_{\phi}(0){\mathcal G}
\left(\frac{||({\phi}^0,{\phi}^1)||^2_{X_1\times X_2}}{E_{\phi}(0)}
\right) \le \int_0^T |\sqrt{a} \dot \phi|^2_H  \,dt
\end{equation}
for any non-identically zero initial data $(\phi^0,\phi^1) \in H_{1/2} \times H$.
\end{description}

Our main results are stated as follows:

\begin{Theorem}\label{main}
Let $\eta>0$ and $T_0>0$ be fixed given real numbers. For any $r\in
(0, \eta)$, we define a function $K_r$ from $(0,r)$ on $[0,\infty)$
by
\begin{equation}\label{Kr}
K_r(\tau)=\int_{\tau}^r \frac{1}{v (f {\mathcal
G}_{\theta})^{-1}(v)}\,dv,
\end{equation}
here ${\mathcal G}_{\theta}= {\mathcal G}\circ x^{\frac{1}{\theta} -1}$. \noindent We
also define
\begin{equation}\label{psir}
\psi_r(z)=z+ K_r\big(f {\mathcal G}_{\theta}\big(\frac{1}{z}\big)\big),
\quad z \ge \frac{1}{(f {\mathcal G}_{\theta})^{-1}(r)}\,.
\end{equation}
Assume $({\bf A1})$ and $({\bf A2})$. Then for non-identically zero initial data $(w^0,w^1) \in H_{1} \times H_{1/2}$, the energy of the strong solution of
\eqref{damped1} satisfies
\begin{equation}\label{general-decay}
E_w(t) \le \beta T (f {\mathcal G}_{\theta})^{-1}
\Big(\frac{1}{\psi_r^{-1}(\frac{t-T}{T_0})}\Big) \ , \quad \mbox{
for } t \mbox{ sufficiently large} \,.
\end{equation}
\end{Theorem}

\begin{Remark}
Suppose further that the function 
\begin{eqnarray*}
h : (0,1) &&\rightarrow \mathbb{R}_+ \\
 x &&\mapsto \frac{1}{x^{\frac{\theta}{1 - \theta}}} \mathcal{G}
\end{eqnarray*}
is increasing on $(0,1)$.

Notice that  
$$
h(\alpha x) \leq h(x), \forall \, \alpha \in (0,1), x \in (0,1),
$$
or equivalently  
$$
\mathcal{G} (\alpha x) \leq \alpha^{\frac{\theta}{1-\theta}} \mathcal{G}(x), \forall \, \alpha \in (0,1), x \in (0,1).
$$
Letting $\alpha$ goes to zero this implies that $\mathcal{G}(0) = 0$ and then $\mathcal{G}(x) > 0$ for all $x > 0$.
In this case the inequality \rfb{ineq-obs} implies, according to \cite[Theorem 2.4]{ammari}, that we have a weak stability for the linear associated problem, i.e., there exists a constant $C > 0$ such that for all $t > 0$ and for all $(w^0,w^1) \in H_1 \times H_\half$ we have that the solution of \rfb{damped1} with $\rho = Id$ satisifes:
$$
E_w(t) \leq C \, \left[{\mathcal G}^{-1} \left(\frac{1}{1+t}\right)\right]^{\frac{\theta}{1-\theta}} \, \left\|(w^0,w^1)\right\|^2_{H_1 \times H_\half}.
$$
\end{Remark}

\medskip

Let ${\mathcal H}: \mathbb{R}_+ \rightarrow \mathbb{R}_+$ such that 
${\mathcal H}$ is continuous, invertible and increasing on $\mathbb{R}_+$. 

\begin{description}
\item[Assumption (A3)]  \
There exist $T, C_T >0$ such that the following observability inequality
is satisfied for the linear conservative system \eqref{conservative} 

\begin{equation}\label{ineq-obsbis}
C_T ||({\phi}^0,{\phi}^1)||^2_{H_1 \times H_\half} \, {\mathcal H}
\left(\frac{E_{\phi}(0)}{||({\phi}^0,{\phi}^1)||^2_{H_1 \times H_\half}}
\right) \le \int_0^T |\sqrt{a} \dot \phi|^2_H  \,dt
\end{equation}
for any non-identically zero initial data $(\phi^0,\phi^1) \in H_{1} \times H_\half$.
\end{description}

By the same way as in Theorem \ref{Kr} we have the following result.

\begin{Theorem} \label{mainbis}
Let $\eta>0$ and $T_0>0$ be fixed given real numbers. For any $r\in
(0, \eta)$, we define a function ${\mathcal K}_r$ from $(0,r)$ on $[0,\infty)$
by
\begin{equation}\label{Krbisbis}
{\mathcal K}_r(\tau)=\int_{\tau}^r \frac{1}{v (f {\mathcal
H})^{-1}(v)}\,dv.
\end{equation}
We also define
\begin{equation}
\label{psirbisn}
\Psi_r(z)=z+ {\mathcal K}_r\big(f {\mathcal H} \big(\frac{1}{z}\big)\big),
\quad z \ge \frac{1}{(f {\mathcal H})^{-1}(r)}\,.
\end{equation}
Assume $({\bf A1})$ and $({\bf A3})$. Then for non-identically zero initial data $(w^0,w^1) \in H_{1} \times H_{1/2}$, the energy of the strong solution of
\eqref{damped1} satisfies
\begin{equation}\label{general-decaybisbis}
E_w(t) \le \beta T (f {\mathcal H})^{-1}
\Big(\frac{1}{\Psi_r^{-1}(\frac{t-T}{T_0})}\Big) \ , \quad \mbox{
for } t \mbox{ sufficiently large} \,.
\end{equation}
\end{Theorem}
\begin{Remark}
\begin{enumerate}
\item
If we suppose in addition that the function $x \mapsto \frac{1}{x} \, {\mathcal H}(x)$ is increasing on $(0,1).$ Then,
the estimate \rfb{general-decaybisbis} is a generalization (to the nonlinear case) of \rfb{nunif2} in Theorem \ref{princ3}. 
\item
The case ${\mathcal H} = Id$ corresponds to the situation treated in \cite[Theorem 1.1]{am-al} (which we can compare to the linear case, i.e., Theorem \ref{princ3}.)
\end{enumerate}
\end{Remark}

\section{Intermediate results}
\medskip

We start by a key Lemma which relies on the optimal-weight convexity
method of~\cite{amo2005} (see also \cite{am-al,jde2010-2,jde2010-1}), so the proof
will be omitted.

\medskip

\begin{Lemma}\label{energy-kinetic}
Assume that $\rho$ and $a$ satisfy the assumption $(A1)$ and that
there exists $r_0>0$ sufficiently small so that the function $R$
defined by \eqref{H} is strictly convex on $[0,r_0^2]$. Let
$(w^0,w^1)\in H_{1} \times H_{1/2}$, non-identically zero, be given and $(\phi^0,\phi^1)=
(w^0,w^1)$ and $w$ and $\phi$ be the respective solutions of
\eqref{damped1} and of \eqref{conservative}. Then the following
inequality holds
\begin{eqnarray}\label{kinetic}
&&\int_0^T f \left(\frac{E_\phi(0)}{||({\phi}^0,{\phi}^1)||^2_{H_1 \times H_{1/2}}} \right)\int_{\Omega}\Big(a(x)|\dot w|^2 + a(x)|\rho(x,\dot w)|^2\Big)\,dx\,dt \nonumber \\
&\leq& c_5 T R^{\star}\left(f \left(\frac{E_\phi(0)}{||({\phi}^0,{\phi}^1)||^2_{H_1 \times H_{1/2}}} \right) \right) \\
&+& c_6\left(f \left(\frac{E_\phi(0)}{||({\phi}^0,{\phi}^1)||^2_{H_1 \times H_{1/2}}}\right) +1 \right) \int_0^T \int_{\Omega}a(x)\rho(x,\dot w)\dot w \,dx\,dt\,,\nonumber
\end{eqnarray}
where
$$
c_5=|\Omega |(1+ c_2^2) \,,\, 
c_6=\Big(\frac{1}{c_1}+ c_2\Big)\,,
$$
\noindent and $|\Omega |=\int_{\Omega}\,d\sigma$, with
$d\sigma=a(.)dx$.
\end{Lemma}

\medskip

The next Lemma compares the localized kinetic damping of the
linearly damped equation with the localized linear and nonlinear
kinetic energies of the nonlinearly damped equation.

\medskip

\begin{Lemma}\label{comparison}
Assume that $\rho \in \mathcal{C}(\overline{\Omega} \times
\mathbb{R};\mathbb{R})$ is a continuous monotone nondecreasing
function with respect to the second variable on $\Omega$ such that
$\rho(.,0)=0$ on $\Omega$. Let $w$ be the solution of
\eqref{damped1} with non-identically zero initial data $(w^0,w^1) \in H_1 \times H_{1/2}$.
Let us introduce $z$ solution of the linear locally damped problem
\begin{equation}
\begin{cases}
\ddot z + A z + a(x) \dot z =0 \,, \label{lineardamped}\\
z(0)=w^0, \dot z(0)=w^1 \,. 
\end{cases}
\end{equation}
\noindent Then the following inequality holds
\begin{equation}\label{linearobs}
\int_0^T \int_{\Omega}a(x)|\dot z|^2 \,dx\,dt \leq 2
\int_0^T\int_{\Omega}\Big(a(x)|\dot w|^2 + a(x)|\rho(x,\dot
w|^2\Big)\,dx\,dt.
\end{equation}
\end{Lemma}

The next Lemma compares the localized observation for the
conservative undamped equation with the localized damping of the
linearly damped equation.

\begin{Lemma}\label{phiz}
Assume that $a \in \mathcal{C}(\overline{\Omega})$, with $a \ge 0$
on $\Omega$. Let $T>0$ be given, then there exists $k_T>0$ (given by
$k_T=8T^2||a||_{L^{\infty}(\Omega)}^2 + 2 \, $) such that for all
$(w^0,w^1) \in H_1 \times H_{1/2}$
\begin{equation}\label{phizobs}
\int_0^T\int_{\Omega}a|\dot \phi|^2\,dx\,dt \leq k_T
\int_0^T\int_{\Omega}a|\dot z|^2\,dx\,dt
\end{equation}
\noindent where $\phi$ is the solution of the conservative equation
\eqref{conservative} with $(\phi^0,\phi^1)=(w^0,w^1)$ and $z$ is the
solution of \eqref{lineardamped}.
\end{Lemma}

\section{Proof of the main results}
The following lemmas will be very useful.
\begin{Lemma}\label{comparison2}
Let $\delta>0$ and $M$ be an increasing and a non-negative function such that the
function defined by $\psi(x)= x - \rho_T M(x)$ is strictly
increasing on $[0,\delta]$, for some positive constant $\rho_T$.
Assume that $\widehat{E}$ is a nonnegative, nonincreasing function
defined on $[0,\infty)$ with $\widehat{E}(0) < \delta$ and
satisfying
\begin{equation}\label{recurrent1}
\widehat{E}((k+1)T) \leq \widehat{E} (kT) - \rho_T M(\widehat{E}
(kT)) \,, \quad \forall \ k \in \mathbb{N} \,.
\end{equation}
\noindent After we consider the sequence $(\widetilde{y_k})_k$
defined by induction as follows:
\begin{equation}\label{ytilde}
\begin{cases}
 \widetilde{y_{k+1}}- \widetilde{y_k} + \rho_T M(\widetilde{y_k})=0\,, k \in \mathbb{N}\,,\\
\widetilde{y_0}=E_0.
\end{cases}
\end{equation}
\noindent Then the following inequality holds
\begin{equation}\label{compEy}
E_k \le \widetilde{y_k} \,,
\end{equation}
here we set
\begin{equation}\label{EkT}
E_k=\widehat{E}(kT) \,, \quad \forall \ k \in \mathbb{N}.
\end{equation}
\end{Lemma}
\begin{proof}
Since the sequence
$(\widetilde{y_k})_k$ satisfies \eqref{ytilde}, so we have
\begin{equation}\label{induction1}
E_{k+1} - \widetilde{y_{k+1}} \leq \psi(E_k)- \psi(\widetilde{y_k})
\,, \quad \forall \ k \in \mathbb{N}\,.
\end{equation}
We prove \eqref{compEy} par induction on $k$. Since $E_0 \le
\widetilde{y_0}$, \eqref{compEy} holds for $k=0$. Assume that
\eqref{compEy} holds at the order $k$. First, we remark that since
$\widehat{E}$ is nonincreasing and thanks to our assumption $E_0 <
\delta$, we have
$$
E_k < \delta \,, \quad \forall \ k \in \mathbb{N} \,.
$$
\noindent Moreover, it is easy to check that the sequence
$(\widetilde{y_k})_k$ is nonincreasing, so that
$$
\widetilde{y_k} \leq \widetilde{y_0}=E_0 < \delta \,, \quad \forall
\ k \in \mathbb{N} \,.
$$
\noindent Thanks to our choice of $\delta$, and since we make the
assumption that $E_k \leq \widetilde{y_k}$, we deduce that
$$
\psi(E_k)- \psi(\widetilde{y_k}) \le 0.
$$
\noindent Using this last estimate in \eqref{induction1}, we deduce
that \eqref{compEy} holds at the order $k+1$.
\end{proof}

\medskip

We now compare the sequence $(\widetilde{y_k})$ obtained using an
Euler scheme to the solution of the associated ordinary differential
equation at time $kT$.

\medskip

\begin{Lemma}\label{comparison3}
Assume the hypotheses of Lemma~\ref{comparison2}. We define $E_k$ as
in \eqref{EkT}. We consider the ordinary differential equation
\begin{equation}\label{diffy}
\begin{cases}
\displaystyle{y^{\prime}(s) + \frac{\rho_T}{T}M(y(s))= 0 \,,  \quad s \ge 0 \,,}\\
y(0)=E_0 \,
\end{cases}
\end{equation}
\noindent and set
\begin{equation}\label{sk}
s_k=kT \,, y_k=y(s_k) \,, \quad \forall \ k \in \mathbb{N}.
\end{equation}
\noindent Then we have for all $k$ in $\mathbb{N}$
\begin{equation}\label{yytilde}
\widetilde{y_k} \le y_k \,,
\end{equation}
\noindent where $(\widetilde{y_k})_k$ is defined by \eqref{ytilde}.
\end{Lemma}
\begin{proof}
We integrate \eqref{diffy} between $s_k$ and $s_{k+1}$ and compare
with the equation satisfied by $\widetilde{y_k}$. Thus we have
\begin{equation}\label{diffyytilde}
y_{k+1} - \widetilde{y_{k+1}} - (y_k - \widetilde{y_k}) +
\frac{\rho_T}{T} \int_{s_k}^{s_{k+1}} \Big(M(y(s)) -
M(\widetilde{y_k})\Big)\,ds=0 \,, \quad \forall \ k \in \mathbb{N}
\,.
\end{equation}
\noindent We prove \eqref{yytilde} by induction on $k$. The property
clearly holds for $k=0$. Assume that it holds at the order $k$.
Since $y$ is nonincreasing, we deduce that $y_k=y(s_k) \le y_0= E_0
< \delta$.  Thus
$$
y(s) \le y_k < \delta \,, \quad \forall \ s \in [s_k,s_{k+1}] \,.
$$
\noindent Since $M$ is nondecreasing, we deduce from
\eqref{diffyytilde} that
$$
\Big(\psi(y_k) - \psi(\widetilde{y_k})\Big)\le  y_{k+1} -
\widetilde{y_{k+1}}.
$$
\noindent Since we assume that \eqref{yytilde} holds at the order
$k$ and since $\psi$ is nondecreasing on $[0,\delta]$, we deduce
$$
0 \le \Big(\psi(y_k) - \psi(\widetilde{y_k})\Big).
$$
\noindent Using this last inequality in the above one, we prove
\eqref{yytilde} at the order $k+1$.
\end{proof}

We deduce from Lemmas~\ref{comparison2} and \ref{comparison3}
the following result.

\begin{Corollary}\label{Eky}
Assume the hypotheses of Lemma~\ref{comparison2} and
Lemma~\ref{comparison3}. Then we have
\begin{equation}\label{comp4}
E_k \le y(s_k) \,, \quad \forall \ k \in \mathbb{N} \,.
\end{equation}
\end{Corollary}

\medskip

The proof of the main result rely on the following abstract theorem
of which proof based on the previous lemmas is given in
\cite{am-al}.
\begin{Theorem}\label{discretcont}
Let $\eta>0$ and $T_0>0$ be fixed given real numbers and let $F$ be
strictly increasing function from  $[0,+\infty)$ onto $[0,\eta)$,
with $F(0)=0$ and $\lim_{y\rightarrow\infty} F(y)=\eta$. For any
$r\in (0, \eta)$, we define a function $K_r$ from $(0,r)$ on
$[0,\infty)$ by
\begin{equation}\label{Krbis}
K_r(\tau)=\int_{\tau}^r \frac{1}{v F^{-1}(v)}\,dv
\end{equation}
\noindent We also define
\begin{equation}\label{psirbis}
\psi_r(z)=z+ K_r\big(F\big(\frac{1}{z}\big)\big), \quad z \ge
\frac{1}{F^{-1}(r)}\,.
\end{equation}
Let $T>0$ and $\rho_{T}>0$ be given. Let $\delta>0$ be such that the
function defined by $x \mapsto x-\rho_T x F^{-1}(x)$ is strictly
increasing on $[0,\delta]$. Assume that $\widehat{E}$ is a
nonnegative, nonincreasing function defined on $[0,\infty)$ with
$\widehat{E}(0) < \delta$ and satisfying
\begin{equation}\label{recurrent1bis}
\widehat{E}((k+1)T) \leq \widehat{E} (kT) \Big(1 -  \rho_T
F^{-1}(\widehat{E} (kT))\Big) \,, \quad \forall \ k \in \mathbb{N}
\,.
\end{equation}
\noindent Then $\widehat{E}$ satisfies the upper estimate
\begin{equation}\label{general-decay2}
\widehat{E}(t) \le T
F\Big(\frac{1}{\psi_r^{-1}(\frac{(t-T)\rho_T}{T_0})}\Big) \ , \quad
\mbox{ for } t \mbox{ sufficiently large}.
\end{equation}
\end{Theorem}
 We repeat the proof for the reader's convenience.
\begin{proof}[Proof of Theorem~\ref{discretcont}] 
We set
\begin{equation}\label{T0}
T_0= \frac{T}{\rho_T} \,,\, r=\widehat{E}(0) \,, M(v)=v F^{-1}(v).
\end{equation}
\noindent Thus the solution $y$ of \eqref{diffy} is characterized as
\begin{equation}\label{carac}
y(t)=K_r^{-1}(\frac{t}{T_0}) \,, \quad \ t \ge 0 \,.
\end{equation}
\noindent On the other hand, we define $E_k$ by \eqref{EkT}. Then,
thanks to \eqref{recurrent1bis}, $E_k$ satisfies 
\begin{equation}
E_{k+1} \leq E_{k} \Big(1 -  \rho_T
F^{-1}(E_{k})\Big) \,, \quad \forall \ k \in \mathbb{N}
\,.
\end{equation} 
Let $l\in \mathbb{N}$ be an arbitrary fixed
integer. We have in particular
$$
E_{k+1+i} - E_{k+i} +  \rho_T M(E_{k+i}) \leq 0\,, \quad \mbox{for }
i=0 \ldots\,, i=l.
$$
\noindent Summing these inequalities from $i=0$ to $i=l$, and using
the fact that $(E_k)_k$ is a nonincreasing sequence whereas $M$ is a
nondecreasing function, we obtain
$$
E_{k+l+1} - E_k + \frac{1}{T_0} (l+1)T M(E_{k+l}) \le 0
$$
\noindent so that
\begin{equation}\label{discrete1}
(l+1) T M(E_{k+l}) \le T_0 E_k \,, \quad \forall \ k ,l \in
\mathbb{N}\,.
\end{equation}
\noindent In particular, we have for any arbitrary $p \in
\mathbb{N}$
\begin{equation}\label{discrete2}
M(E_p) \le \frac{T_0}{T} \inf_{l \in \{0,\ldots,p\}}\Big(
\frac{E_{p-l}}{l+1}\Big) \,.
\end{equation}
\noindent Now thanks to Corollary~\ref{Eky} and to \eqref{carac}, we
have
$$
E_i \le y_i=K_r^{-1}\Big(\frac{iT}{T_0}\Big) \,, \quad \forall \ i
\in \mathbb{N}.
$$
\noindent Using this last relation in \eqref{discrete2}, we deduce
that
\begin{equation}\label{discrete3}
M(E_p) \le \frac{T_0}{T} \inf_{l \in \{0,\ldots,p\}}\Big(
\frac{K_r^{-1}\big(\frac{(p-l)T}{T_0}\big)}{l+1}\Big) \,.
\end{equation}
\noindent Let now $t \ge T$ be given and $p \in \mathbb{N}$ be the
unique integer so that $t \in [pT,(p+1)T)$. Let $\theta \in (0,t-T]$
be arbitrary and $l \in \mathbb{N}$ be the unique integer so that $
\theta \in [lT, (l+1)T)$. Then, thanks to \eqref{discrete3} and by
construction, we have
$$
\displaystyle{M(\widehat{E}(t)) \le M(E_p) \le \frac{T_0}{T} \inf_{l
\in \{0,\ldots,p\}}\Big(
\frac{K_r^{-1}\big(\frac{(p-l)T}{T_0}\big)}{l+1}\Big) \,,}
$$
\noindent and
$$
\displaystyle{K_r^{-1}\left(\frac{(p-l)T}{T_0}\right) \le
K_r^{-1}\left(\frac{t-\theta - T}{T_0}\right) \,.}
$$
\noindent We deduce that
$$
\displaystyle{M(\widehat{E}(t)) \le \frac{T}{\theta} K_r^{-1}\Big(
\frac{t-T-\theta}{T_0}\Big) \,, \quad \forall \ \theta \in
(0,t-T]\,.}
$$
\noindent Using the fact that $M$ is strictly increasing, we obtain
$$
\displaystyle{\widehat{E}(t) \le T M^{-1} \left(\inf_{\theta \in (0,
(t-T)]} \left( \frac{1}{\theta}
K_r^{-1}\left(\frac{t-T-\theta}{T_0}\right)\right) \right)}.
$$
\noindent Let now $t>0$ be fixed for the moment and put
$\gamma_t(\theta)=\frac{1}{\theta}
K_r^{-1}\left(\frac{t-T-\theta}{T_0}\right)$. Thus $\theta^
*$ is a critical point of $\gamma_t$ if and only if it satisfies the
relation:
$$
K_r^{-1}\left (\frac{t-T-\theta^*}{T_0}\right ) + \frac{\theta^
*}{T_0 K_r^{'}K_r^{-1}\big (\frac{t-T-\theta^*}{T_0}\big )}=0.
$$
Hence $\theta^*$ is a critical point of $\gamma_t$ if and only if it
solves the equation

$$
K_r^{-1}\left (\frac{t-T-\theta^*}{T_0}\right ) = \frac{\theta^*}{T_0}
M \left (K_r^{-1}\left (\frac{t-T-\theta^*}{T_0}\right)\right).
$$

Using the definition of $M$, we deduce that $\theta^
*$ is a critical point of $\gamma_t$ if and only if it satisfies the following equation:
$$
\frac{T_0}{\theta^*}=
F^{-1}\left(K_r^{-1}\left(\frac{t-T-\theta^*}{T_0}\right)\right)
$$
Hence $\theta^
*$ is a critical point of $\gamma_t$ if and only if it verifies the following equation:
$$
\psi_r \left(\frac{\theta^*}{T_0}\right)=\frac{t-T}{T_0},
$$
and we obtain
$$
\displaystyle{\widehat{E}(t) \le  T
F\left(\frac{1}{\psi_r^{-1}\big(\frac{t-T}{T_0} \big)}\right) \,, \quad
\forall \ t \ge T \,.}
$$
\noindent So that \eqref{general-decay2} is proved.
\end{proof}

\medskip

\begin{proof}[Proof of Theorem~\ref{main}]
\begin{eqnarray*}
&&\int_0^T
f\left(\frac{E_\phi(0)}{||({\phi}^0,{\phi}^1)||^2_{H_1 \times H_{1/2}}} \right)\int_{\Omega}\Big(a(x)|\dot
w|^2 + a(x)|\rho(x,\dot w)|^2\Big)\,dx\,dt \\
&&\geq c_T
f \left(\frac{E_\phi(0)}{||({\phi}^0,{\phi}^1)||^2_{H_1 \times H_{1/2}}} \right)\int_0^T\int_{\Omega}
a(x)|\dot \phi|^2\,dx\,dt\\
&&\geq c_T \,
f \left(\frac{E_\phi(0)}{||({\phi}^0,{\phi}^1)||^2_{H_1 \times H_{1/2}}} \right)||({\phi}^0,{\phi}^1)||^2_{H_{1/2}
\times H} \, {\mathcal G} \left(\frac{||({\phi}^0,{\phi}^1)||^2_{X_1
\times X_2}} {||({\phi}^0,{\phi}^1)||^2_{H_{1/2} \times H}} \right)\\
&&\geq c_T \,
f \left(\frac{E_\phi(0)}{||({\phi}^0,{\phi}^1)||^2_{H_1 \times H_{1/2}}} \right)||({\phi}^0,{\phi}^1)||^2_{H_{1/2}
\times H} \, {\mathcal G}
\left(\left(\frac{||({\phi}^0,{\phi}^1)||^2_{H_{1/2} \times
H}}{||({\phi}^0,{\phi}^1)||^2_{H_1 \times H_{1/2}}} \right)^{\frac{1}{\theta
}-1}\right),
\end{eqnarray*}
here $c_T = \frac{1}{2k_T}$.\\

\noindent Since
$$R^* \left(f \left(\frac{E_\phi(0)}{||({\phi}^0,{\phi}^1)||^2_{H_1 \times H_{1/2}}} \right) \right)=
\frac{E_\phi(0)}{\beta
||({\phi}^0,{\phi}^1)||^2_{H_1 \times H_{1/2}}}f \left(\frac{E_\phi(0)}
{||({\phi}^0,{\phi}^1)||^2_{H_1 \times H_{1/2}}} \right),$$ this together with
(\ref{kinetic}) and the definition of the weight function $f$ lead
to:
\begin{equation}\label{obs3}
C_T E_\phi(0)f(\widehat{E}_w (0)) {\mathcal G}((\widehat{E}_w
(0))^{\frac{1}{\theta}-1}) \leq \frac{C_T}{\beta} \widehat{E}_w
(0)f(\widehat{E}_w (0))+ C_7(E_w (0)-E_w (T)), \,
\end{equation}
\noindent where we put
$\widehat{E}_w (0)= \frac{E_\phi(0)}{||({\phi}^0,{\phi}^1)||^2_{H_1 \times H_{1/2}}}$. Moreover,
\begin{equation}\label{obs37}
C_T \widehat{E}_w (0)f(\widehat{E}_w (0)) \mathcal{G} ((\widehat{E}_w
(0))^{\frac{1}{\theta}-1}) \leq C_T \frac{\widehat{E}_w
(0)f(\widehat{E}_w (0))}{\beta ||({\phi}^0,{\phi}^1)||^2_{H_1 \times H_{1/2}}}+
C_7(\widehat{E}_w (0)-\widehat{E}_w (T)). \,
\end{equation}
gives
\begin{equation}
\widehat{E}_w (T)\leq \widehat{E}_w (0)[1-\big (C'_T
{\mathcal G}((\widehat{E}_w (0))^{\frac{1}{\theta} -1})-\frac{C_8 T}{\beta
||({\phi}^0,{\phi}^1)||^2_{H_1 \times H_{1/2}}}\big )f(\widehat{E}_w (0)).
\end{equation}
Choose $\beta$ so that $\big (C'_T {\mathcal G}((\widehat{E}_w
(0))^{\frac{1}{\theta}-1})-\frac{C_8 T}{\beta
||({\phi}^0,{\phi}^1)||^2_{H_1 \times H_{1/2}}}\big )> C''_T {\mathcal G}((\widehat{E}_w
(0))^{\frac{1}{\theta}-1})$.

\noindent Hence
\begin{equation}
\widehat{E}_w (T)\leq \widehat{E}_w (0)[1-\big (C''_T
{\mathcal G}((\widehat{E}_w (0)^{\frac{1}{\theta}-1}))f(\widehat{E}_w (0))].
\end{equation}
Make use of Theorem~\ref{discretcont}, the proof is complete.
\end{proof}

\begin{proof}[Proof of Theorem \ref{mainbis}]
The proof is a simple adaptation of the proof of Theorem \ref{main}.
\end{proof}

\section{Some applications}\label{appls}
We give applications of Theorems \ref{main} and \ref{mainbis}. In the next result, we denote by $C$ a positive constant depending on $E(0)$ and $T$.
Also, we give only the expression of $g$ in a right neighbourhood of $0$, since as long as
$g$ has a linear growth at infinity, the asymptotic behavior of the
energy depends only on the behavior of $g$ close to $0$.

\medskip

We assume that $\rho$ and $a$ satisfy assumption ${\bf{(A1)}}$.
We assume that there exists $T>0$ such that the solution of \eqref{conservative} satisfies the weak observability inequality
\eqref{ineq-obs} for example 1 below and the assumption \eqref{ineq-obsbis} for examples 2 and 3. Then, we have the following results:

\subsection{Example 1}
Let $g$ be given by $g(x)=x^p$, $p>1$ on $(0,r_0]$. 
Then the energy of solution of \eqref{damped1} satisfies the estimate
$$
E_w (t)\leq C \, \left(x \mapsto x^{\frac{p-1}{2}}{\mathcal G}_{\theta}(x)\right)^{-1}\left(\frac{1}{t+1} \right),
$$
for $t$ sufficiently large and for all any non-identically zero initial data $(w^0,w^1) \in H_{1} \times H_{1/2}$.

\subsection{Example 2}
Let $g$ and ${\mathcal H}$ are given by $g(x)=x^p$, $p>1$ on $(0,r_0]$. Then the energy of solution of (\ref{damped1}) satisfies the estimate
$$E_w (t)\leq \, \left(x \mapsto x^{\frac{p-1}{2}}{\mathcal H}(x) \right)^{-1} \left(\frac{1}{t+1}\right),$$
for $t$ sufficiently large and for all any non-identically zero initial data $(w^0,w^1) \in H_{1} \times H_{1/2}$.\\
\underline{Particular case}: For ${\mathcal H}(x) = \exp \left(-\frac{C}{x^{\frac{1}{p}}} \right)$, $C, p>0$ the last estimate becomes
$$E_w (t)\leq \, \frac{C}{(\ln(1+t))^{p}}.$$

\subsection{Example 3}
Let $g$ be given by $g(x)= x^3 \, \exp \left(-\frac{1}{x^2} \right)$. Then the energy of solution of \eqref{damped1} satisfies the estimate
$$
E_w (t)\leq C \, \left(x \mapsto \exp \left(-\frac{1}{x}\right)\mathcal{H}(x)\right)^{-1}\left(\frac{1}{1+t}\right),
$$
for $t$ sufficiently large and for all any non-identically zero initial data $(w^0,w^1) \in H_{1} \times H_{1/2}$.\\
\underline{Particular cases}: For ${\mathcal H}(x)= \exp \left(-\frac{C}{x^{\frac{1}{p}}} \right)$, $C, p>0$ the last estimate becomes :
$$E_w (t)\leq \, \frac{C}{\ln(1+t)}, \mbox{ for } p\geq 1,$$
and
$$E_w (t)\leq \, \frac{C}{(\ln(1+t))^{p}}, \mbox{ for } p < 1.$$

\subsection{Example 4}
Here we consider the following initial and boundary problem:
\begin{equation}\label{dampedw}
\begin{cases}
u_{tt} - \Delta u + 
a(x)\rho(x, u_t) = 0, \, \quad  (x,t) \in \Omega \times
(0, + \infty), \\[0.8mm]
u = 0 \,, \quad \mbox{on } \partial \Omega \times (0,+ \infty),\\[0.8mm]
u(x,0) = u^0(x), \, u_t(x,0) = u^1(x) \,,  \quad \mbox{on } \Omega,
\end{cases}
\end{equation}
where $\rho$ and $a$ satisfy assumption $({\bf A 1})$ and 
$\Omega$ is a convex bounded open set of $\mathbb{R}^N$ of class $\mathcal{C}^2$.

\noindent
In this case, we have: 
$$
A = - \, \Delta : D(A) \subset H = L^2(\Omega) \rightarrow L^2(\Omega), \, H_1 = D(A) = H^2(\Omega) \cap H^1_0 (\Omega), 
$$
$
H_{1/2} = H^1_0(\Omega)
$
and $A$ is a selfadjoint operator satisfying (\ref{eq:opA}).

Moreover the conservative equation (\ref{conservative}) becomes in this case:

\begin{equation}\label{conservativepb}
\begin{cases}
\phi_{tt} -  \Delta \phi = 0 \,, 
\quad \Omega \times (0,+\infty)\,,\\
\phi=0,  \quad \partial \Omega \times (0,+\infty), \\ 
\phi(x, 0) = \phi^0(x), \phi_t(x, 0) = \phi^1(x),  
\quad \Omega. 
\end{cases}
\end{equation}

According to \cite{phunghdr} we show that the observability inequality is given by
\begin{Proposition}
For all $\beta \in ]0,1[$ there exists $T$ and $c_T >0$ such that the following observabilty inequality holds:
\begin{eqnarray}\label{obs}
&&||({\phi}^0,{\phi}^1)||^2_{\left[H^2(\Omega) \cap H^1_0(\Omega) \right] \times H^1_0(\Omega)} \, \exp {\left[ - c_T \, \left(
\frac{||({\phi}^0,{\phi}^1)||_{\left[H^2(\Omega) \cap H^1_0(\Omega) \right] \times H^1_0 (\Omega)}}{||({\phi}^0,{\phi}^1)||_{H^1_0(\Omega)\times L^2 (\Omega)}}
\right)^{1/\beta}\right]} \nonumber \\  
&&  \hspace{1cm}\leq\int_0^T |\sqrt{a} \dot \phi|^2_H  \,dt \,, \quad
\end{eqnarray}
for all non-identically zero initial data
$(\phi^0,\phi^1) \in \left[H^2(\Omega) \cap H^1_0(\Omega) \right] \times H^1_0 (\Omega).$
\end{Proposition}
We remark here that we have \eqref{ineq-obsbis} for ${\mathcal H}(x) = \exp(- \frac{c_T}{x^{1/2\beta}}), \, \forall \, x > 0.$

\medskip

Thus according to Theorem \ref{mainbis} we have the following stabilization result for the nonlinear damped wave equation as in \cite{kim,daoulatli,bel}.

\begin{Theorem} 
We suppose that $meas(supp \, a) \neq 0$. Then, the energy of solution of (\ref{dampedw}) satisfies for all $\beta \in ]0,1[$ the estimate:
\begin{equation}\label{general-decaybis}
E_w(t) \le  \frac{C}{(\ln ( 1 + t))^{2\beta}} \ , \quad \mbox{
for } t \mbox{ sufficiently large}
\end{equation}
and for all any non-identically zero initial data $(u^0,u^1) \in \left[H^2(\Omega) \cap H^1_0(\Omega)\right] \times H^1_0 (\Omega)$.
\end{Theorem}

\section{Appendix: Weak stabilization of linear evolution systems} \label{prel'}
Let $H$ be a Hilbert space with the norm $||.||_H$, and let $A:{\mathcal D}(A) \subset H \rightarrow H$ be a self-adjoint, positive and boundedly invertible operator. We also introduce the scale of Hilbert spaces
$H_{\alpha}$, as follows\m: for every $\alpha\geq 0$,
$H_{\alpha}=\mathcal{D}(A^{\alpha})$, with the norm $\|z
\|_{\alpha}=\|A^\alpha z\|_{H}$. The space $H_{-\alpha},$ is defined
by duality with respect to the pivot space $H$ as follows\m:
$H_{-\alpha} =H_{\alpha}^*,$ for $\alpha>0$.

Let the bounded linear operator $B : U \rightarrow H$, where $U$ is another Hilbert space
which will be identified with its dual.

The system we consider is described by
\begin{equation}
\label{damped}
\ddot{w}(t) + A w(t) + BB^* \dot{w}(t) = 0, \, w(0) = w_0,  \dot{w}(0) = w_1, \, t \in [0,\infty),
\end{equation}
The system (\ref{damped}) is well-posed:\\
\noindent
For $(w_0,w_1) \in H_\half \times H,$ the problem \rfb{damped}  admet a unique solution
$$w \in C([0,\infty);H_\half \times H)$$
such that $B^* \dot{w}(\cdot)\in L^2_{loc}(0,+\infty;U)$. Moreover, $w$ satisfies the energy estimate,
for all $t\geq 0$
\begin{equation}
\|(w_0,w_1)\|^2_{H_\half \times H}-
\|(w(t),\dot{w}(t))\|^2_{H_\half \times H} = 2 \, \int_0^t
\left\|B^* \dot{w}(s)\right\|_{U}^2\, ds \, .
\label{ESTEN}
\end{equation}
For (\ref{ESTEN}) we remark that the mapping $t \mapsto
\|(w(t),\dot{w}(t))\|^2_{H_\half \times H}$ is non-increasing. \\
\noindent
Consider the initial value problem :
\begin{equation}
\ddot\varphi(t) + A \varphi(t) = 0, 
\label{eq3}\ee
\be \varphi(0) = \varphi_0, \dot{\varphi}(0) = \varphi_1. 
\label{eq4}
\end{equation}
It is well known that (\ref{eq3})-(\ref{eq4}) is well posed in $H_1 \times H_\half$ and in $H_\half \times H$. \\
\noindent
Now, we consider the unbounded linear operator
\begin{equation}
\label{opa}
{\mathcal A}_d : {\mathcal D}({\mathcal A}_d) \subset H_\half \times H \rightarrow H_\half \times H , \,
{\mathcal A}_d =  
\left( \begin{matrix} I & 0 \cr - A & - BB^* \end{matrix} \right),
\end{equation}
where
\[
{\mathcal D}({\mathcal A}_d) = H_1 \times H_\half.
\]
Let ${\mathcal H}: \mathbb{R}_+ \rightarrow \mathbb{R}_+$ such that 
${\mathcal H}$ is continuous, invertible, increasing on $\mathbb{R}_+$ and suppose that the function $x \mapsto \frac{1}{x} \, {\mathcal H}(x)$  is increasing on $(0,1).$ \\

\noindent
In the case of non exponential decay in the energy space we have the
explicit decay estimate valid for regular initial data, which is a simple adaptation of \cite[Theorem 2.4]{ammari}.
\begin{Theorem} \label{princ3} 
Assume that the function ${\mathcal H}$ 
satisfies the assumptions above. Then the following assertion holds true:\\

\noindent
If for all non-identically zero initial data
$(\varphi_0,\varphi_1) \in H_1 \times H_\half$ we have
\be 
\int_{0}^{T} ||B^* \dot{\varphi} (t)||^2_{U} \, dt \geq C \, 
||(\varphi_0,\varphi_1)||^2_{H_1 \times H_\half} \, {\mathcal H} \left(\frac{||(\varphi_0, \varphi_1)||^2_{H_\half \times H}}
{||(\varphi_0,\varphi_1)||^2_{H_1 \times H_\half}} \right),
\label{nunif1}
\ee
for some constant $C > 0$ then there exists a constant $C_1 > 0$ such that for all $t > 0$ and for all non-identically zero initial data
$(w_0,w_1) \in H_1 \times H_\half$ we have
\begin{equation}
\left\|(w(t),\dot{w}(t))\right\|^2_{H_\half \times H} \leq C_1 \,{\mathcal H}^{-1} \left(\frac{1}{1 + t} \right) \, 
||(w_0,w_1)||^2_{H_1 \times H_\half}.
\label{nunif2}\end{equation}
\end{Theorem}
\begin{Remark}
In the case where ${\mathcal H} = Id$ the observability inequality \rfb{nunif1} is equivalent to the exponential stability of \rfb{damped}, see \cite[Theorem 2.2]{ammari}.
\end{Remark}
\begin{proof}
We suppose \rfb{nunif1}, which implies that there exist $C, T > 0$ such that for all non-identically zero initial data
$(w^0,w^1) \in H_1 \times H_\half$ we have
\[ 
\int_{0}^{T} ||B^* \dot{\varphi}(t)||^2_{U} \, dt \geq C \, 
||(w_0,w_1)||^2_{H_\half \times H} \, {\mathcal H} \left( 
\frac{||(w_0,w_1)||^2_{H_\half \times H}}{||(w_0,w_1)||^2_{H_1 \times H_\half}} \right).
\]
By applying \cite[Lemma 4.1]{ammari} we obtain that the solution $w(t)$
of \rfb{damped} satisfies the following inequality 
$$
\int_0^T ||B^* \dot{w}(t)||_{U}^2 dt
\ge C \, ||(w_0,w_1)||^2_{H_\half \times H} \, {\mathcal H} \left( 
\frac{||(w_0,w_1)||^2_{H_\half \times H} }{||(w_0,w_1)||^2_{H_1 \times H_\half}} \right).
$$
Relation above and \rfb{ESTEN} imply the existence of a constant $K>0$
such that  
\begin{eqnarray}
||(w(T),\dot{w}(T))||^2_{H_\half \times H} &&\leq 
||(w_0,w_1)||^2_{ H_\half \times H} \nonumber \\ &&
- K \, ||(w_0,w_1)||^2_{H_\half \times H} \, {\mathcal H} \left( 
\frac{||(w_0,w_1)||^2_{H_\half \times H}}{||(w_0,w_1)||^2_{H_1 \times H_\half}} \right).
\label{PREINT}
\end{eqnarray}
By using the fact that the function
$t \mapsto ||(w(t),\dot{w}(t))||^2_{H_\half \times H}$ 
is nonincreasing, the function ${\mathcal H}$ is increasing and  relation  (\ref{PREINT}) we obtain
the existence of a constant $K_1>0$ such that
\begin{eqnarray}
||(w(T),\dot{w}(T))||^2_{H_\half \times H} &&\leq
||(w_0,w_1)||^2_{H_\half \times H} \nonumber \\
&&- K_1 \, ||(w_0,w_1)||^2_{H_\half \times H} \, {\mathcal H} \left( 
\frac{||(w(T),\dot{w}(T))||^2_{H_\half \times H}}{||(w_0,w_1)||^2_{H_1 \times H_\half}} \right).
\label{INTP}
\end{eqnarray}
Estimate (\ref{INTP}) remains valid in successive intervals 
$[kT, (k+1)T]$,
so, we have
\begin{eqnarray*}
&&||(w((k+1)T),\dot{w}((k+1)T))||^2_{H_\half \times H}\leq 
 ||(w(kT),\dot{w}(kT))||^2_{H_\half \times H}\\&&-K_1 \, ||(w(kT),\dot{w}(kT)||^2_{H_\half \times H} \, {\mathcal H} \left(
\frac{||(w((k+1)T),\dot{w}((k+1)T)||^{2}_{H_\half \times H}}
{||(w(kT),\dot{w}(kT))||_{H_1 \times H_\half}^{2}} \right).
\end{eqnarray*}
Since ${\mathcal A}_d$ generates a semigroup of contractions in 
${\mathcal D}({\mathcal A}_d)$,
relations above imply the existence of a constant $K_2>0$ such that
\begin{eqnarray}
&&||(w((k+1)T),\dot{w}((k+1)T))||^2_{H_\half \times H}\le
||(w(kT),\dot{w}(kT)||^2_{H_\half \times H} \nonumber
\\ 
&& -K_2 \, ||(w(kT),\dot{w}(kT))||^2_{H_\half \times H} \, {\mathcal H} \left(
\frac{||(w((k+1)T),\dot{w}((k+1)T))||^2_{H_\half \times H}}
{||(w_0,w_1)||_{H_1 \times H_\half}^{2}} \right),
\label{SUC}
\end{eqnarray}
If we adopt now the notation
\be 
{\mathcal E}_k = {\mathcal H} \left(\frac{||(w(kT),\dot{w}(kT))||^2_{H_\half \times H}}
{||(w_0,w_1)||^2_{H_1 \times H_\half}} \right),
\label{NOTEN}
\ee
the inequality \rfb{SUC} implies
\be
\label{1}
\frac{||(w((k+1)T),\dot{w}((k+1)T))||^2_{H_\half \times H}}
{||(w(kT),\dot{w}(kT))||^2_{H_\half \times H}} \, \frac{{\mathcal E}_k}{{\mathcal E}_{k + 1}} \,
{\mathcal E}_{k + 1} \le {\mathcal E}_k - K_2 \, {\mathcal E}_k \, {\mathcal E}_{k + 1}.
\ee
Since, the function $t \rightarrow ||(w(t),\dot{w}(t))||^2_{H_\half \times H}$ 
is nonincreasing and 
the function ${\mathcal H}$ is increasing, relation \rfb{1} implies 
\be
\label{2}
\frac{||(w((k+1)T),\dot{w}((k+1)T))||^2_{H_\half \times H}}
{||(w(kT),\dot{w}(kT))||^2_{H_\half \times H}} \, \frac{{\mathcal E}_k}{{\mathcal E}_{k + 1}} \,
{\mathcal E}_{k + 1} \le {\mathcal E}_k - K_2 \, {\mathcal E}_{k + 1}^2.
\ee
According to \rfb{NOTEN}, relation \rfb{2} gives,
\begin{eqnarray}
&&\frac{\frac{1}{\frac{||(w(kT),\dot{w}(kT))||^{2}_{H_\half \times H}}{||(w_0,w_1)||^{2}_{H_1 \times H_\half}}} \, {\mathcal H} \left(
\frac{||(w(kT),\dot{w}(kT))||^{2}_{H_\half \times H}}{||(w_0,w_1)||^{2}_{H_1 \times H_\half}}\right)}
{\frac{1}{\frac{||(w((k + 1)T),\dot{w}((k + 1)T))||^{2}_{H_\half \times H}}{||(w_0,w_1)||^{2}_{H_1 \times H_\half}}} \, {\mathcal H} \left(
\frac{||(w((k + 1)T),\dot{w}((k + 1)T))||^{2}_{H_\half \times H}}{||(w_0,w_1)||^{2}_{H_1 \times H_\half}}\right)}
\, {\mathcal E}_{k+1} \nonumber \\
&& \le {\mathcal E}_k -
K_2 \, {\mathcal E}^{2}_{k+1}.\label{3}
\end{eqnarray}
Relation (\ref{3}) combined with that the function $x \mapsto 
\frac{1}{x} \, {\mathcal H}(x)$ is increasing in 
$(0,1)$, gives
\be  {\mathcal E}_{k+1}\le {\mathcal E}_k -
K_2{\mathcal E}^{2}_{k+1},\ \forall k\ge 0.
\label{REC}
\ee
By applying \cite[Lemma 5.2]{ammarisicon} and
using relation (\ref{NOTEN}) we obtain
the existence of a constant $M>0$ such that
$$||(w(kT),\dot{w}(kT))||^2_{H_\half \times H}\le
{\mathcal H}^{-1} \left(\frac{M}{k+1} \right) \, 
||(w_0,w_1)||^2_{H_1 \times H_\half},\
\forall k\ge 0,$$
which obviously implies \rfb{nunif2}. 

\end{proof}

\subsection*{Example}
We consider the following initial and boundary problem:
\begin{equation}\label{dampedww}
\begin{cases}
u_{tt} - \Delta u + 
a(x) \, u_t = 0, \,, \quad  (x,t) \in \Omega \times
(0, + \infty) \\[0.8mm]
u = 0 \,, \quad \mbox{on } \partial \Omega \times (0,+ \infty),\\[0.8mm]
u(x,0) = u^0(x), \, u_t(x,0) = u^1(x) \,,  \quad \mbox{on } \Omega,
\end{cases}
\end{equation}
where $\Omega$ is a convex bounded open set of $\mathbb{R}^N$ of class $\mathcal{C}^2$ and $a \in {\mathcal C} (\overline{\Omega})$ with $a \geq 0$ on $\Omega$ and as in assumption $({\bf A 1})$.

\noindent
In this case, we have: 
$$
A = - \, \Delta : {\mathcal D}(A) = H_1 \subset L^2(\Omega) \rightarrow L^2(\Omega), \, H_1 = {\mathcal D}(A) = H^2(\Omega) \cap H^1_0 (\Omega), 
$$
$
H_{1/2} = H^1_0(\Omega), U = L^2(\Omega)$ and $B z = B^* z = \sqrt{a} z, \forall \, z \in L^2(\Omega)$. 

Moreover the conservative equation (\ref{conservative}) becomes in this case:

\begin{equation}\label{conservativepbw}
\begin{cases}
\phi_{tt} -  \Delta \phi = 0 \,, 
\quad \Omega \times (0,+\infty)\,,\\
\phi=0,  \quad \partial \Omega \times (0,+\infty), \\ 
\phi(x, 0) = u^0(x), \phi_t(x, 0) = u^1(x),  
\quad \Omega. 
\end{cases}
\end{equation}

According to \cite{phunghdr} we show that the observability inequality is given by:
\begin{Proposition}
For all $\beta \in ]0,1[$ there exist $T, c_T>0$ such that the following observabilty inequality holds:
\begin{eqnarray}\label{obsww}
 \left\|(u^0,u^1)\right\|^2_{\left[H^2(\Omega) \cap H^1_0(\Omega)\right] \times H^1_0(\Omega)} \, \exp {\left[- \, c_T \left(
\, \frac{\left\|(u^0,u^1)\right\|_{\left[H^2(\Omega) \cap H^1_0(\Omega)\right] \times H^1_0(\Omega)}}{||(u^0,u^1)||_{H^1_0(\Omega)\times L^2 (\Omega)}} \right)^{1/\beta} \right]}&& \nonumber \\ \le 
\int_0^T \int_\Omega a(x) \, |\phi_t (x,t)|^2 \, dx  \,dt \,, \quad
\end{eqnarray}
for all any non-identically zero initial data
$(u^0,u^1) \in \left[H^2_0(\Omega) \cap H^1_0(\Omega)\right] \times H^1_0(\Omega).$
\end{Proposition}
We remark here that we have \eqref{nunif1} for ${\mathcal H}(x) = \exp(- \frac{c_T}{x^{1/2\beta}}), \, \forall \, x > 0$.
Thus according to Theorem \ref{princ3} we have the following stabilization result for the linear wave equation which extends the result obtained by \cite[Lebeau]{leb} (with a resolvent method).

\begin{Theorem} For all $\beta \in ]0,1[$, there exists a constant $C > 0$ such that 
for all any non-identically zero initial data $(u^0,u^1) \in \left[H^2(\Omega) \cap H^1_0(\Omega)\right] \times H^1_0 (\Omega)$ the energy of the solution of (\ref{dampedww}) satisfies the estimate
\begin{equation}
\label{general-decaybisw}
\left\|(u(t),\dot{u}(t))\right\|_{H^1_0 (\Omega) \times L^2(\Omega)} \le \frac{C}{(\ln (1 + t))^{\beta}} \left\|(u^0,u^1)\right\|_{\left[H^2(\Omega) \cap H^1_0(\Omega)\right] \times H^1_0 (\Omega)}, \, t > 0. 
\end{equation}
\end{Theorem}

\section*{Acknowledgments}
The authors would like to thank Kim Dang Phung and Luc Robbiano for very fruitful discussions concerning \rfb{obsww}.


\begin{thebibliography}{99}
\bibitem{jde2010-2} F. Alabau-Boussouira, New trends towards lower energy estimates and optimality for nonlinearly damped vibrating systems, {\em J. of Differential Equations,}  {\bf 249} (2010), 1145--1178.
\bibitem{jde2010-1} F. Alabau-Boussouira, A unified approach via convexity for optimal energy decay rates of finite and infinite dimensional vibrating damped systems with applications to semi-discretized vibrating damped systems, {\em J. of Differential Equations,}  {\bf 248} (2010), 1473--1517.
\bibitem{amo2005}  F. Alabau-Boussouira,  Convexity and weighted integral inequalities for energy decay rates of nonlinear dissipative hyperbolic systems, {\em Appl. Math. and Optimization}, {\bf 51} (2005), 61--105.
\bibitem{am-al} F. Alabau-Boussouira and K. Ammari, Sharp energy estimates for nonlinearly locally damped PDE's via observability for the associated undamped system, {\em J. Funct. Anal,} {\bf 260} (2011), 2424--2450.
\bibitem{ammarisicon} K. Ammari and M. Tucsnak, Stabilization of Bernoulli-Euler beams by means of a pointwise feedback force, {\em
SIAM Journal on Control and Optimization.}, {\bf 39} (2000), 1160-1181.
\bibitem{ammari} K. Ammari and M. Tucsnak, {\em Stabilization of second order evolution equations by a class of unbounded feedbacks,} {\em ESAIM COCV.}, {\bf 6} (2001), 361--386.
\bibitem{ammariniciase} K. Ammari and S. Nicaise, {\em Stabilization of elastic systems by collocated feedback}, vol. 2124, Springer-Verlag, Berlin, 2015.
\bibitem{barbu} V. Barbu, {\em Nonlinear differential equations of monotone types in Banach spaces,} Springer Monographs in Mathematics. Springer, New York, 2010.
\bibitem{bel} M. Bellassoued, {\em Decay of solutions of the wave equation with arbitrary localized nonlinear damping,} 
J. Differential Equations, {\bf 211} (2005), 303--332.
\bibitem{daoulatli} M. Daoulatli, {\em Rate of decay of solutions of the wave equation with arbitrary localized nonlinear damping,} 
Nonlinear Anal, {\bf 73} (2010), 987--1003. 
\bibitem{leb} G. Lebeau, Equation des ondes amorties, {\em Algebraic and geometric methods in mathematical physics (Kaciveli, 1993),} 73-109, Math. Phys. Stud., 19, Kluwer Acad. Publ., Dordrecht, 1996.
\bibitem{kim} K.-D. Phung, {\em Decay of solutions of the wave equation with localized nonlinear damping and trapped rays,}  Math. Control Relat. Fields 
{\bf 1} (2011), 251--265.
\bibitem{phunghdr} K.-D. Phung, {\em Observation et stabilisation d'ondes :
g\'eom\'etrie et co\^ut du contr\^ole,} Hdr, 2007.
\bibitem{rob} L. Robbiano, {\em Fonction de co\^ut et contr\^ole des solutions des \'equations hyperboliques. Asymptotic Anal.}, {\bf 10} (1995) 95--115.
\bibitem{triebel} H. Triebel, {\em Interpolation theory, function spaces, differential operators}, North Holland, Amsterdam, 1978.

\end{thebibliography}
\end{document}